\documentclass[12pt,a4paper,reqno]{amsart}
\usepackage{amsfonts, amsmath, amssymb, amsthm}
\usepackage{graphicx}
\usepackage[dvipsnames]{xcolor}
\usepackage[margin=1in]{geometry}

\newtheorem{thm}{Theorem}
\newtheorem{lemma}[thm]{Lemma}
\newtheorem{corollary}{Corollary}[thm]
\theoremstyle{remark}
\newtheorem*{rem*}{Remark}

\numberwithin{equation}{section}

\title[Generalisations of the Landau--Gonek Theorem]{Generalisations of the Landau--Gonek Theorem and Applications to Mean Values of Zeta}
\author[B. Durkan]{Benjamin Durkan}
\address{Department of Mathematics, The University of Manchester, Oxford Road, Manchester, M13 9PL}
\email{benjamin.durkan@postgrad.manchester.ac.uk}
\author[C. Hughes]{Christopher Hughes}
\address{Department of Mathematics, University of York, York, YO10 5GH, United Kingdom}
\email{christopher.hughes@york.ac.uk}
\author[A. Pearce-Crump]{Andrew Pearce-Crump}
\address{School of Mathematics, Fry Building, Woodland Road, Bristol, BS8 1UG, United Kingdom}
\email{andrew.pearce-crump@bristol.ac.uk}
\date{}

\begin{document}

\begin{abstract}
  The Landau--Gonek Theorem evaluates $X^\rho$ summed over the non-trivial zeros of the Riemann zeta function. Their result shows great sensitivity to the arithmetic nature of $X$. We prove a related result concerning the sum of $\chi(\rho) X^\rho$ over the zeros of zeta, where $\chi(s)$ is the term arising in the functional equation for the zeta function. Again, this result depends deeply on whether $X$ is an integer or not. We show the result splits into three cases, depending on whether $X$ is smaller than $T$, about the same size as $T$, or bigger than $T$. The reason this result is useful is that it easily permits the calculation of discrete moments of the Riemann zeta function via the approximate functional equation.  As an application of this result, we provide an alternative proof of Shanks' conjecture.
\end{abstract}

\maketitle

\section{Introduction}
Since Riemann's 1859 memoir \cite{Riemann}, the link between primes and the non-trivial zeros of the Riemann zeta function $\zeta(s)$ has been a central topic within analytic number theory. A common theme is to consider sums of the form 
\begin{equation}\label{eq:sumfrho}
\sum_{0<\gamma\le T} f(\rho)
\end{equation}
for a function $f(s)$, where $\rho=\beta+i\gamma$ is a non-trivial zero of the Riemann zeta function $\zeta(s)$.

Landau~\cite{landau} established the following result for $f(s)=X^{s}$, a variation on a sum that occurs naturally when studying the Prime Number Theorem.

\begin{thm}[Landau, 1911]
Given fixed $X>1$ we have \begin{equation*}
    \sum_{0<\gamma\le T}X^\rho=-\frac{T}{2\pi}\Lambda(X)+O(\log T),
\end{equation*}
as $T\to\infty$, where $\Lambda(n)$ is the von Mangoldt function.
\end{thm}

This result tells us that the main term in the sum under consideration only appears when $X$ is a prime or a prime power, since $\Lambda(n)$ is only supported when $n$ equals a prime power. Such a sum, relating the non-trivial zeros to the primes, is referred to as an explicit formula. Gonek~\cite{gonek1985formula, gonek} made Landau's Theorem uniform in both $X$ and $T$.

\begin{thm}[Gonek, 1985]\label{gonekuniform}
   Uniformly for $X,T>1$ we have 
   \begin{align*}
        \sum_{0<\gamma\le T}X^{\rho}&=-\frac{T}{2\pi}\Lambda(X)+O(X\log(2XT)\log\log(3X))\\
        &+O\left(\log X\min\left(T,\frac{X}{\langle X\rangle}\right)\right)+O\left(\log(2T)\min\left(T,\frac{1}{\log X}\right)\right),
    \end{align*}
    where $\langle X\rangle$ is the distance from $X$ to the closest prime power (different from $X$).
\end{thm}

When we restrict $X$ to be an integer, Gonek also notes that the last two error terms are subsumed by the first two error terms. He also highlights the highly different behaviour between the case when $X$ is an integer and when $X$ is an arbitrary real number. 

Gonek's result was further generalised by Fujii~\cite{FujiiLandau1, FujiiLandau2} who found lower order terms in the expansion.

So far, we have been restricted to the case that $X>1$. Gonek noted that we can also consider the case $0 < X <1$ by noting that the functional equation for the zeta function, given by
\begin{equation}\label{eq:FE}
\zeta(s)=\chi(s)\zeta(1-s)
\end{equation}
implies that if $\rho$ is a zeta zero, so too is $1-\overline{\rho}$ (in fact they have the same imaginary part). Then $\sum_{0<\gamma\le T} X^{-\rho}=\sum_{0<\gamma\le T} X^{\overline{\rho-1}}$. Using this observation leads to the following corollary. 

\begin{corollary}[Gonek, 1993]\label{cor:lg}
Uniformly for $X>1$ and $T>1$, we have
\begin{align}\label{eq:SumNToTheRho}
        \sum_{0<\gamma\le T}X^{-\rho}&=-\frac{T}{2\pi} \frac{\Lambda(X)}{X} + O(\log (2XT)\log\log (3X)) \notag \\
        &+O\left(\log X\min\left(\frac{T}{X},\frac{1}{\langle X\rangle}\right)\right)+O\left(\log(2T)\min\left(\frac{T}{X},\frac{1}{X\log X}\right)\right), 
    \end{align}
    where $\langle X\rangle$ is the distance from $X$ to the closest prime power (different from $X$).
\end{corollary}

Again, when we restrict $X$ to be an integer, Gonek also notes that the last two error terms are subsumed by the first two error terms.

Another natural function to consider in \eqref{eq:sumfrho} is $f(s) = \zeta ' (s)$. Indeed, it was a conjecture by Shanks~\cite{Shanks61} (based on numerical evidence from Haselgrove~\cite{haselgrove1960tables}) that, when summed over the non-trivial zeta zeros, $\zeta'(\rho)$ is real and positive in the mean. This is a surprising conjecture, since both $\zeta'(s)$ and the non-trivial zeros are complex. This conjecture was first proved by Conrey, Ghosh and Gonek~\cite{CGG88} where they showed that
\[
\sum_{0<\gamma\le T}\zeta'(\rho)=\frac{T}{4\pi}\left(\log\frac{T}{2\pi}\right)^2+O(T\log T),
\]
from which Shanks' conjecture follows immediately. Trudgian~\cite{TRUDGIAN20102635} also gave an alternative proof of this result.

This has been refined considerably since then. Fujii~\cite{fujii,Fujii12} established the full asymptotic, and in 2022 Hughes and Pearce-Crump~\cite{discreteMVT} sharpened the error term under the Riemann Hypothesis (RH). 
\begin{thm}[Fujii, 1994]\label{thm:shanks}
\begin{multline*}
\sum_{0 < \gamma \leq T} \zeta' (\rho) =
 \frac{T}{4 \pi} \left(\log\frac{T}{2 \pi} \right)^2 +(\gamma_0-1)\frac{T}{2\pi} \log \left(\frac{T}{2\pi} \right)  + (1-\gamma_0-\gamma_0^2-3\gamma_1)\frac{T}{2 \pi}  +\mathcal{E}(T)
\end{multline*}
where \begin{equation*}
    \mathcal{E}(T)=\begin{cases}
        O\left(T\exp(-C\sqrt{\log T})\right) &\text{unconditionally, for some }C>0\\ O(T^{1/2}(\log T)^{13/4}) &\text{under RH, due to Hughes--Pearce-Crump, 2022}
    \end{cases}
\end{equation*}
\end{thm}
Here $\gamma_0$ and $\gamma_1$ are the Stieltjes constants given by the following coefficients of the Laurent expansion of $\zeta(s)$ around $s=1$
\begin{equation} \label{eq:Laurent}
    \zeta(s)=\frac{1}{s-1}+\gamma_0 - \gamma_1 (s-1) + \frac{\gamma_2}{2} (s-1)^2 + \dots.
\end{equation}

One can similarly deal with the higher derivatives with $f(s) = \zeta^{(\nu)} (s)$ in \eqref{eq:sumfrho}. In 2011, Kaptan, Karabulut and Y{\i}ld{\i}r{\i}m~\cite{kaptan} established the leading-order asymptotic for this function.

\begin{thm}[Kaptan--Karabulut--Y{\i}ld{\i}r{\i}m, 2011]\label{higherderivs}
    We have \begin{equation}
        \sum_{0<\gamma\le T}\zeta^{(\nu)}(\rho)=\frac{(-1)^{\nu+1}}{\nu+1} \frac{T}{2\pi} \left(\log\frac{T}{2\pi}\right)^{\nu+1}+O(T(\log T)^{\nu}).
    \end{equation}
\end{thm}

In words, this result says that $\zeta^{(\nu)} (s)$ is real and positive/negative in the mean depending on whether $\nu$ is odd/even, a result the second two authors of this paper have coined the generalised Shanks' conjecture.

Hughes and Pearce-Crump~\cite{discreteMVT} subsequently established the full asymptotic with all lower order terms made explicit, and with power-savings on the error term.

In 2023, Hughes, Martin, and Pearce-Crump~\cite{hughes2023heuristic} devised a simple heuristic which gives the correct main term for this sum, but with an error term that dominates the main term. The starting point of their heuristic is to use the approximate formula
\begin{equation}\label{eq:approxform}
\zeta^{(\nu)}(\sigma+it)=(-1)^{\nu}\sum_{m\le t}\frac{(\log m)^{\nu}}{m^{\sigma+it}}+O(t^{-\sigma}(\log t)^{\nu}),
\end{equation}
valid for $\sigma>0$ and $t$ large. Then, upon swapping the order of summation and summing over the zeta zeros using Corollary \ref{cor:lg}, one recovers immediately the correct leading term of Theorem \ref{higherderivs}. However, they showed the error terms dominate, so the argument is heuristic, not rigorous. This heuristic may be viewed as taking the approximate functional equation for the derivatives of the zeta function (see Section~\ref{sect:shanks2}) and placing the entire weight on the first piece, and then summing over the zeta zeros. The aim of this paper is to develop a result allowing us to use both pieces of the approximate functional equation, making the heuristic rigorous.

Finally, a combination of some of the functions considered above was considered by Fujii~\cite{Fujii12} where he considered $f(s) = \zeta ' (s) X^s$, and by Pearce-Crump~\cite{MR4868785} where he considered $f(s) = \zeta ^{(\nu)}(s) X^s$. In both of these instances, the behaviour of such sums changing depending on whether $X$ is an integer or a general real number is observed. In both cases the asymptotics are more complicated than those written out above --- we refer the reader to their papers for the full results.

\section{Results}

In this paper we prove Theorem \ref{mainthm}, where we establish uniform asymptotics for \eqref{eq:sumfrho} with $f(s)=\chi(s)X^{s}$, where $\chi(s)$ is the factor from the functional equation \eqref{eq:FE} for zeta. 

Let 
\begin{equation*}
   S(X,T)=\sum_{T<\gamma\le 2T}\chi(\rho)X^{\rho}.
   \end{equation*}
We prove an asymptotic for the sum $S(X,T)$ which is uniform in both $X$ and $T$. Theorem~\ref{mainthm} may be viewed as an oscillatory analogue of Theorem \ref{gonekuniform}.

\begin{thm}\label{mainthm}
Uniformly for $X \geq 1 $ and $T >  1$ we have
	\begin{equation*}
S(X,T) = 
	\begin{cases}
		\displaystyle -X\sum_{\frac{T}{2\pi X} < n\le\frac{T}{\pi X}}\Lambda(n)e^{2\pi Xni} + E(X,T) & \text{ if } X \leq \frac{T}{2\pi}  \\[2ex]
		\displaystyle X (\log X) e^{2\pi Xi} + E(X,T) & \text{ if } \frac{T}{2\pi}  < X \leq \frac{T}{\pi} \\[2ex]
		\displaystyle -X\sum_{\frac{\pi X}{T}\le n < \frac{2\pi X}{T}}\frac{\Lambda(n)}{n}e^{\frac{2\pi Xi}{n}} +  E(X,T) & \text{ if } X > \frac{T}{\pi} 
	\end{cases}
\end{equation*}
where
\begin{multline}
E(X,T) = O\left(T^{1/2} (\log T)^2 \right)  + O\left( \frac{X^{1+1/\log T} (\log T)^2}{T^{1/2}}\right) 
+ O\left(\frac{T^{3/2}\log T}{|T - 2\pi X| + T^{1/2}}\right) \\ + O\left(\frac{T^{3/2}\log T}{|T - \pi X| + T^{1/2}}\right) . \label{eq:ErrorTermRH}
\end{multline}
\end{thm}

\begin{rem*}
	Note that the last two error terms are the same size as the main terms at the jumps, when $X=\frac{T}{2\pi} $ and $X=\frac{T}{\pi} $. Also note that the two sums over $n$ become empty at precisely those jumps.
\end{rem*}

We illustrate the theorem in Figure~\ref{fig1} by plotting $S(X,T)$ in the complex plane for all $T$ above the first zeta zero and below $300,000$ and $X=2000$. It is a discrete set consisting of 1,466,163 points, and is beautifully intricate. The points in blue are for $T<\pi X$, those in red are for $\pi X < T \leq 2\pi X$ (red), and those in green are for $T > 2\pi X$, clearing showing that $S(X,T)$ has three distinct parts, all captured by our theorem.
\begin{center}
\begin{figure}[ht]
\includegraphics{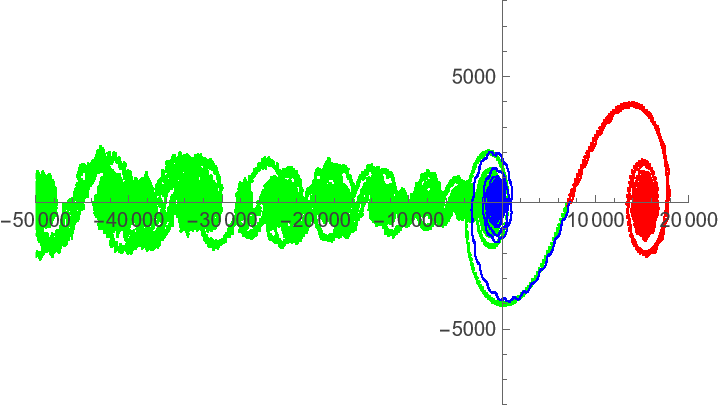}
\caption{The complex values of $S(X,T)$, plotted for $T<300,000$ with $X=2000$. Note that the $\Re$- and $\Im$-axes have very different scalings}\label{fig1}
\end{figure}
\end{center}

Our Theorem immediately recovers and sharpens the error term in a theorem due to Pearce-Crump~\cite{andrewthesis, Pearce-Crump24} where he considered $X=1$. This result was previously known with a worse error term in \cite{conrey1990zeros, KaraYil11}, so the results in this paper carries on the progression of improving the error term in this special case. We state this result as follows.

\begin{corollary}\label{cor:X=1}
    We have 
    \begin{equation*}
        \sum_{0<\gamma\le T} \chi(\rho) = -\frac{T}{2\pi} +
        \begin{cases}
        O\left(Te^{-a\sqrt{\log T}}\right) & \text{unconditionally}\\ 
        O(T^{1/2}(\log T)^2) & \text{under RH}
        \end{cases}
    \end{equation*}
    for some $a>0$.
\end{corollary}

\begin{proof}[Proof of Corollary \ref{cor:X=1}]
Setting $X=1$ in Theorem~\ref{mainthm} and summing over dyadic intervals we have
\begin{align*}
\sum_{0<\gamma\le T} \chi(\rho) &= \sum_{k=1}^{\lfloor \frac{\log (T/(2\pi))}{\log 2} \rfloor} S(1,\frac{1}{2^k} T) \\
&= -\sum_{n\leq \frac{T}{2\pi}} \Lambda(n) + O\left(T^{1/2} (\log T)^2\right)\\
&= -\frac{T}{2\pi} +
\begin{cases}
	O\left(T e^{-a\sqrt{\log T}}\right) & \text{unconditionally}\\ O(T^{1/2}(\log T)^2) & \text{under RH}
\end{cases}
\end{align*}
The truncation of the dyadic sum is chosen so that both $2^{-k} T$ is larger than the lowest zero for all $k$ and also so that we only ever require the first case in Theorem~\ref{mainthm}. Since $X$ is an integer in this corollary, $e^{2\pi i X n} = 1$ for all $n$. We remark that the error term being conditional/unconditional arises solely from our evaluation of the sum of the von Mangoldt function. 
\end{proof}

Theorem \ref{mainthm} permits other positive integers, not just $X=1$, to be summed over all zeros up to height $T$, with a similar result.
\begin{corollary}\label{mainthmwithinteger}
	Let $X \in \mathbb{N}$ be such that $X=o(T)$ be an integer. Then
	 \begin{equation*}
		\sum_{0<\gamma\le T}  \chi(\rho) X^\rho = -\frac{T}{2\pi} + O(X\log X) +
		\begin{cases}
			O\left(Te^{-a\sqrt{\log T}}\right) & \text{unconditionally}\\ 
			O(T^{1/2}(\log T)^2) & \text{under RH}
		\end{cases}
	\end{equation*}
    for some $a>0$.
	\end{corollary}

\begin{proof}[Proof of Corollary~\ref{mainthmwithinteger}]
There are three regions when performing the dyadic summation, corresponding to the three cases in Theorem~\ref{mainthm}. For $k$ such that $\frac{1}{2^k}T \geq 2\pi X$, we have a contribution from the first case in that theorem, and the $n$-sums combine to yield
\[
-X \sum_{n\leq \frac{T}{2\pi X}} \Lambda(n) = -\frac{T}{2\pi} + 
\begin{cases}
	O\left(T e^{-a\sqrt{\log T}}\right) & \text{unconditionally}\\ O(T^{1/2}(\log T)^2) & \text{under RH}
\end{cases}
\]

There is exactly one value of $k$ such that $X \in \left(\frac{1}{2\pi}  2^{-k} T, \frac{1}{\pi} 2^{-k} T\right]$, and that will cause the middle case of the theorem to contribute $X \log X$.

And finally, the $k$ such that $\frac{1}{2^k}T < 2 \pi X$ will yield contributions from the third case, where the $n$-sums combine to yield
\[
-X \sum_{n < X} \frac{\Lambda(n)}{n} e^{2 \pi i X / n} \ll X \sum_{n <X} \frac{\Lambda(n)}{n} = O(X \log X)
\]
where, as before, we choose to stop at the dyadic interval when $k$ is the largest integer such that $2^{-k} T > 2\pi$.

When summed over all the dyadic intervals, the errors contribute
\[
\sum_{k=1}^{\lfloor \frac{\log T/(4\pi)}{\log 2} \rfloor} E(X,2^{-k} T) = O(T (\log T)^2) + O(X \log X) 
\]
where the last two terms for $E(X,2^{-k} T)$ in \eqref{eq:ErrorTermRH} potentially dominate for only one $k$, when $2^{-k} T \approx 2\pi X$, say, when it yields $O(X \log X)$.

Unfortunately for $T$ so small, that error term now dominates the main term contributions from the second and third cases of the theorem. In order to not lose these interesting subsidiary main terms, in the main theorem we evaluate the sum of zeros between $T$ and $2T$, rather than over all zeros up to $T$.
\end{proof}

As two further corollaries, we will use Theorem \ref{mainthm} to deduce new proofs for Theorems \ref{thm:shanks} and \ref{higherderivs} in Sections \ref{sect:shanks1} and \ref{sect:shanks2}, respectively. Briefly, we follow the idea of Hughes, Martin and Pearce-Crump~\cite{hughes2023heuristic}, but rather than use the approximate formula \eqref{eq:approxform}, we use the full approximate functional equation \eqref{eq:AFE} for the derivative of zeta. Our result enables us to control how much weight is put into each part of the approximate functional equation, and this, when combined with Theorem \ref{mainthm}, turns their heuristic into a full proof.

\section*{Acknowledgments}
This forms part of the first author's MSc by Research thesis, \cite{Durkan}. The third author acknowledges support from the Heilbronn Institute for Mathematical Research.

\section{Overview of paper}
In Section \ref{sect:proofofthm}, we prove a key result using the method of stationary phase. Specifically, we calculate the integral
\[
J(\sigma,r,T) = \int_{T}^{2T} \chi(\sigma+it)r^{it}dt
\]
which will enable us to evaluate $S(X,T)$. 

We then prove the main result of this paper, Theorem \ref{mainthm}. To do this, we begin by using Cauchy's Theorem to write \begin{equation*}
    S(X,T) = \sum_{T<\gamma\le 2T}\chi(\rho)X^{\rho}=\frac{1}{2\pi i}\int_{\mathcal{R}}\frac{\zeta'}{\zeta}(s)\chi(s)X^s\ ds,
\end{equation*}
where $\mathcal{R}$ is a suitable contour containing all the required non-trivial zeros of the zeta function. We will show that the contribution from the horizontal sides of this integral can be absorbed into an error term, while the vertical segments both contribute to the main terms. For the right-hand side of the contour, we can expand the zeta functions as their Dirichlet series and then apply the result in our lemma concerning $J(\sigma,r,T)$. We apply a standard trick using the functional equation for the zeta function to map the left-hand side of the contour onto the right-hand side. We then follow a similar method to that used on the original right-hand side. 

In Section \ref{sect:shanks1} we demonstrate the utility of Theorem \ref{mainthm} by applying it to derive the full unconditional asymptotic in Theorem \ref{thm:shanks}, which gives a new proof of Shanks' conjecture. 

In Section \ref{sect:shanks2} we outline how one can use Theorem \ref{mainthm} to obtain the asymptotic that proves the generalised Shanks' conjecture. We obtain the leading-order result, given by Theorem \ref{higherderivs}. This section, together with the calculations in the previous section, completes our proof of showing that the heuristic in \cite{hughes2023heuristic} can be made rigorous by using the approximate functional equation.

\section{Proof of Theorem \ref{mainthm}}\label{sect:proofofthm}
We start this section by proving the following lemma used to prove Theorem \ref{mainthm}.

Let
\[
J(\sigma,r,T) = \int_{T}^{2T} \chi(\sigma+it)r^{it}dt.
\]

\begin{lemma}\label{lem_JsigmaRAB}
    For large $T$, uniformly for $-1\leq \sigma \leq 2$ we have 
\begin{equation}\label{lem2variants}
    J(\sigma,r,T )=
    \begin{cases}
        2\pi r^{1-\sigma}e^{2\pi ir}+E(\sigma,r,T)&\text{ if } T < 2\pi r \leq 2T\\ E(\sigma,r,T)&\text{ if } 2\pi r\le T \text{ or }2\pi r > 2T,
    \end{cases}
\end{equation}
where
\begin{equation*}
    E(\sigma,r,T)=O(T^{1/2-\sigma})+O\left(\frac{T^{3/2-\sigma}}{|T -2\pi r|+T^{1/2}}\right)+O\left(\frac{T^{3/2-\sigma}}{|2T-2\pi r|+T^{1/2}}\right) .
\end{equation*}
\end{lemma}

\begin{proof}
    This is essentially due to Gonek~\cite[Lemma 2]{Gonek1984}. We use Stirling's approximation to expand $\chi(\sigma+it)$ for fixed $\sigma$ and $t\ge 1$ as
    \begin{equation}\label{stirling}
        \chi(\sigma+it)=\left(\frac{t}{2\pi}\right)^{\frac{1}{2}-\sigma-it}e^{it+i\frac{\pi}{4}}\left(1+O\left(\frac{1}{t}\right)\right).
    \end{equation}
    This allows us to write our integral as
    \begin{equation*}
        J(\sigma,r,T)=e^{i \pi / 4} \int_{T}^{2T}\left(\frac{t}{2\pi}\right)^{\frac{1}{2}-\sigma}\exp\left[-it\log\left(\frac{t}{2\pi er}\right)\right]\ dt+O(T^{1/2-\sigma}).
    \end{equation*}

By \cite[Lemma 2]{Gonek1984},  we are able to write (after taking the complex conjugate of what is shown there) 
\begin{equation}\label{lem2variant}
    e^{i\pi/4}\int_{T}^{2T} \left(\frac{t}{2\pi}\right)^{\frac{1}{2}-\sigma}\exp\left[-it\log\left(\frac{t}{2\pi er}\right)\right]\ dt =
    \begin{cases}
        2\pi r^{1-\sigma}e^{2\pi ir}+E(\sigma,r,T)\\ 
        E(\sigma,r,T),
    \end{cases}
\end{equation}
where in the first case $T<2\pi r \leq 2T$, and in the second case $2\pi r\le T$ or $2\pi r > 2T$, and where
\begin{equation}\label{equation:gonek_lem_3_error}
    E(\sigma,r,T)=O(T^{1/2-\sigma})+O\left(\frac{T^{3/2-\sigma}}{|T -2\pi r|+T^{1/2}}\right)+O\left(\frac{T^{3/2-\sigma}}{|2T-2\pi r|+T^{1/2}}\right) .
\end{equation}
This completes the proof.
\end{proof}

To prove Theorem \ref{mainthm} we use Cauchy's Theorem to write 
\begin{equation*}
    \sum_{T<\gamma\le 2T}\chi(\rho)X^{\rho}=\frac{1}{2\pi i}\int_{\mathcal{R}}\frac{\zeta'}{\zeta}(s)\chi(s)X^s\ ds,
\end{equation*}
where $\mathcal{R}$ is the positively oriented rectangular contour with vertices $c+ i T, c+2iT, 1-c+2iT, 1-c+ i T$, where $c=1+1/\log T$.

We shall restrict $T$ so that both $|T-\gamma|\gg\frac{1}{\log T}$ and $|2T-\gamma|\gg\frac{1}{\log T}$ for all $\gamma$, that is, the contour is always a distance $\gg\frac{1}{\log T}$ from the ordinate $\gamma$ of any zero. This  incurs a small error term which we deal with now. It is well-known that there are $O(\log T)$ zeros between $T$ and $T+1$ (for example, \cite[Section 9.2]{titchmarsh}), and so by \eqref{stirling} we see that we incur a cost bounded by
\begin{multline*}
    O\left(\chi(\beta+i T)X^{\beta}\log T\right) + O\left(\chi(1-\beta+i T)X^{1-\beta}\log T\right) \\
    = O\left(T^{1/2} \left(\frac{X}{T}\right)^{\beta} \log T \right) + O\left(\frac{X}{T^{1/2}} \left(\frac{T}{X}\right)^{\beta} \log T \right),
\end{multline*}
where $\beta=\max\{\Re(\rho): T\le \Im(\rho)\le T+1\}$, and we consider both the zero at $\beta+i\gamma$ and $1-\beta+i\gamma$, as the dominant contribution changes depending on whether $X\ll T$ or $X\gg T$.  Clearly the strength of the zero-free region we may assume will play a role in the strength of the bound we obtain. If we use the trivial bound $\beta < 1$ we see this is
\[
O\left(\frac{X\log T}{T^{1/2}}\right) + O\left(T^{1/2}\log T\right).
\]
The errors considered here will be shown to be absorbed into the errors $E(X,T)$ in \eqref{eq:ErrorTermRH}.

Returning to the Cauchy integral, we will consider each piece of the contour in turn. We write
\begin{align}
   \frac{1}{2\pi i}\int_{\mathcal{R}}\frac{\zeta'}{\zeta}(s)\chi(s)X^s ds &=\frac{1}{2\pi i}\left(\int_{c+i T}^{c+2iT}+\int_{c+2iT}^{1-c+2iT}+\int_{1-c+2iT}^{1-c+i T}+\int_{1-c+i T}^{c+i T}\right)\frac{\zeta'}{\zeta}(s)\chi(s)X^s ds \notag \\
   &=I_1+I_2+I_3+I_4. \label{eq:SumZerosAsCauchyInt}
\end{align}

We begin by bounding the two horizontal segments before handling the vertical segments, which require a more careful approach.

\begin{lemma}\label{lemtopofcontour}
    The integrals along the top and bottom of the contour are uniformly bounded for $X\ge 1, T>1$ by
    \begin{equation*}
        I_2, I_4 = O\left(T^{1/2} (\log T)^2\right) + O\left( \frac{X^{1+1/\log T} (\log T)^2}{T^{1/2}}\right) 
    \end{equation*}
\end{lemma}
\begin{proof}
    By Gonek~\cite[p. 126]{Gonek1984}, if $T$ is such that $|T-\gamma|\gg \frac{1}{\log T}$ for any zero ordinate $\gamma$, then the bound
\begin{equation*}
\frac{\zeta '}{\zeta}(\sigma + i T) =O\left( (\log T)^{2}\right)
\end{equation*}
holds uniformly for $-1\leq \sigma \leq 2$. We also have from \eqref{stirling} that
\begin{equation}\label{stirling_absolutevalue}
\left| \chi(\sigma + i T) \right| = \left(\frac{T}{2\pi}\right)^{\frac12 -\sigma} \left(1+O\left(\frac{1}{T}\right)\right) .
\end{equation}

Applying these bounds and trivially estimating the resulting integral by the maximum of the integrand (which is maximised at $\sigma=1-c$ if $X\ll T$, and at $\sigma=c$ if $X \gg T$), we have
\begin{equation*}
I_2 =O\left( (\log T)^2 \left(T^{1/2} \left( \frac{T}{X} \right)^{\frac{1}{\log T}} + \frac{X}{T^{1/2}} \left( \frac{X}{T} \right)^{\frac{1}{\log T}} \right)\right)
\end{equation*}

Now, since $\left(\frac{T}{X}\right)^{\frac{1}{\log T}}<T^\frac{1}{\log T}=e$ we deduce that 
\begin{equation*}
        I_2=O\left(T^{1/2} (\log T)^2\right) + O\left( \frac{X^{1+1/\log T} (\log T)^2}{T^{1/2}}\right)
\end{equation*} 
as required.
\end{proof}

Note that in Lemma \ref{lemtopofcontour} the first term dominates if $X\ll T$ and the second term dominates if $X \gg T$. We see also that these two terms collectively dominate the error coming from restricting $|T-\gamma|\gg\frac{1}{\log T}$. This gives the first two terms in the error, $E(X,T)$, given in \eqref{eq:ErrorTermRH}.

\begin{lemma}\label{lem:RHS}
    The integral along the right-hand side of the contour in \eqref{eq:SumZerosAsCauchyInt} is given by
    \begin{align*}
    I_1 &= \frac{1}{2\pi}\int_{T}^{2T}\frac{\zeta'}{\zeta}(c+it)\chi(c+it)X^{c+it}\ dt \\
    &=-X\sum_{\frac{\pi X}{T} \leq n < \frac{2\pi X}{T}}\frac{\Lambda(n)}{n}e^{\frac{2\pi Xi}{n}}+O\left(\frac{X \log T}{T^{1/2}} \right).
\end{align*}
with the error uniform for $X \geq 1$ and $T>1$.
\end{lemma}

\begin{rem*}
This result is non-trivial only for $X \gg T$, otherwise the sum over $n$ is empty.
\end{rem*}

\begin{proof}
Since $c>1$ we are to the right of $\Re(s)=1$ and so we can use the Dirichlet series \begin{equation*}
    \frac{\zeta'}{\zeta}(c+it)=-\sum_{n=1}^{\infty}\frac{\Lambda(n)}{n^{c+it}}.
\end{equation*}
Swapping the sum and integral transforms $I_1$ into 
\begin{equation*}
    I_1=-\frac{1}{2\pi}\sum_{n=1}^{\infty}\Lambda(n)\left(\frac{X}{n}\right)^c \int_{T}^{2T}\left(\frac{X}{n}\right)^{it}\chi(c+it)\ dt.
\end{equation*}

We now apply Lemma \ref{lem_JsigmaRAB} to the integral in this expression, with $\sigma=c$ and $r=X/n$. Since the main term only arises when $2\pi r\in (T,2T]$, we are forced into the regime $\pi X/T \leq n < 2\pi X/T$ and so 
\begin{align*}
    I_1&=-\frac{1}{2\pi}\sum_{n=1}^\infty \Lambda(n) \left(\frac{X}{n}\right)^c J\left(c,\frac{X}{n},T\right) \\
    &=-X\sum_{\frac{\pi X}{T} \leq n< \frac{2\pi X}{T}} \frac{\Lambda(n)}{n} e^{2\pi i X / n}+\sum_{n=1}^\infty \Lambda(n)\left(\frac{X}{n}\right)^c E\left(c,\frac{X}{n},T\right).
\end{align*}

Observe that the first piece occurs as a main term in Theorem \ref{mainthm}. Now we establish an upper bound for the error term,
\begin{multline*}
    \sum_{n=1}^\infty \Lambda(n)\left(\frac{X}{n}\right)^c E\left(c,\frac{X}{n},T\right) = O\left(  X^c \sum_{n=1}^\infty \frac{\Lambda(n)}{n^c} T^{-1/2}  \right) \\
    + O\left(  X^c \sum_{n=1}^\infty \frac{\Lambda(n)}{n^c} \frac{T^{1/2}}{|T-2\pi X/n|+T^{1/2}} \right).
\end{multline*}

Since $c=1+1/\log T$, the series $\sum \Lambda(n)n^{-c}=-\zeta'/\zeta(c) = O(\log T)$ and so we can easily bound the first error term by
\begin{equation*}
   O\left( \frac{X^{1+1/\log T} \log T} {T^{1/2}} \right).
\end{equation*}

For the second error term, using $\sum_{n\leq X} \Lambda(n) \ll X$, this is bounded by
\begin{equation*}
    X^c \int_1^\infty  n^{-1-1/\log T} \frac{T^{1/2}}{|T-2\pi X/n|+T^{1/2}} \, dn \ll  \frac{X}{T^{1/2}} \int_0^{2\pi X / T}  y^{-1+1/\log T} \frac{1}{|1-y|+T^{-1/2}} \, dy
\end{equation*}
coming from letting $y=2\pi X / (T n)$. Splitting into regions around the relevant majorants of the integral, and letting the upper limit be infinity, we can bound this by
\begin{multline*}
\frac{X}{T^{1/2}} \int_0^{1/2} y^{-1+1/\log T} dy + \frac{X}{T^{1/2}} \int_{1/2}^{3/2} \frac{1}{|y-1|+T^{-1/2}} dy + \frac{X}{T^{1/2}} \int_{3/2}^\infty \frac{1}{y^2} dy \\
\ll \frac{X \log T}{T^{1/2}} + \frac{X \log T}{T^{1/2}} + \frac{X}{T^{1/2}} .
\end{multline*}
Note that since 
\[
\int_0^{2\pi/T} y^{-1+1/\log T} dy \gg \log T 
\]
the contribution from $y\approx 0$ is a positive proportion of the total bound, so replacing the upper limit from $2\pi X / T$ with $\infty$ does not cause an overestimate of the total error.
\end{proof}

\begin{lemma}\label{lem_lefthand}
    The integral along the left-hand vertical segment in \eqref{eq:SumZerosAsCauchyInt} is 
    \begin{equation*}
    I_3 = 
    \begin{cases}
    \displaystyle -X\sum_{\frac{T}{2\pi X} < n \leq \frac{T}{\pi X}}\Lambda(n) e^{2\pi i X n} + E'(X,T) & \text{ if } X \leq \frac{T}{2\pi}  \\[1ex]
    \displaystyle X (\log X) e^{2\pi Xi} + E'(X,T) & \text{ if } \frac{T}{2\pi}  < X \le \frac{T}{\pi}  \\[1ex]
    \displaystyle E'(X,T) & \text{ if } X > \frac{T}{\pi} 
    \end{cases}
\end{equation*}
where
\[
E'(X,T) = O(T^{1/2}\log T )  + O\left(\frac{T^{3/2}\log T}{|T - 2\pi X| + T^{1/2}}\right)  + O\left(\frac{T^{3/2}\log T}{|2T - 2\pi X| + T^{1/2}}\right) .
\]
\end{lemma}

To approach Lemma \ref{lem_lefthand} we start by taking the logarithmic derivative of functional equation for zeta \eqref{eq:AFE} to obtain
\begin{align*}
    I_3&=-\frac{1}{2\pi}\int_{T}^{2T}\frac{\zeta'}{\zeta}(1-c+it)\chi(1-c+it)X^{1-c+it}\ dt\\
    &=-\frac{1}{2\pi}\int_{T}^{2T}\left(-\frac{\zeta'}{\zeta}(c-it)+\frac{\chi'}{\chi}(1-c+it)\right)\chi(1-c+it)X^{1-c+it}\ dt\\
    &=\frac{1}{2\pi}\int_{T}^{2T}\frac{\zeta'}{\zeta}(c-it)\chi(1-c+it)X^{1-c+it}\ dt - \frac{1}{2\pi}\int_{T}^{2T}\chi'(1-c+it) X^{1-c+it}\ dt\\ 
    &=I_{3,1} + I_{3,2}.
\end{align*}

We shall deal with $I_{3,1}$ and $I_{3,2}$ separately, and it is clear that Lemma \ref{lem_lefthand} follows from Lemmas~\ref{lem_I31} and~\ref{lem_I32}, below.

\begin{lemma}\label{lem_I31}
    Uniformly for $X \geq 1$ and $T>1$ we have
    \begin{align*}
        I_{3,1} &= \frac{1}{2\pi}\int_{T}^{2T} \frac{\zeta'}{\zeta}(c-it)\chi(1-c+it)X^{1-c+it}\ dt\\
        &=-X\sum_{\frac{T}{2\pi X} < n \leq \frac{T}{\pi X}}\Lambda(n)e^{2\pi i X n} + O(T^{1/2}\log T).
    \end{align*}
\end{lemma}

\begin{proof}[Proof of Lemma \ref{lem_I31}]
By using the Dirichlet series for $\frac{\zeta'}{\zeta}(s)$ we are able to write
\begin{align*}
    I_{3,1} &= \frac{1}{2\pi}\int_{T}^{2T} \frac{\zeta'}{\zeta}(c-it)\chi(1-c+it)X^{1-c+it}\ dt\\
    &=-\frac{1}{2\pi} \int_{T}^{2T} \left(\sum_{n=1}^{\infty}\frac{\Lambda(n)}{n^{c-it}}\right) \chi(1-c+it)X^{1-c+it}\ dt   \\
&= -  \frac{X^{1-c}}{2\pi} \sum_{n=1}^{\infty}\frac{\Lambda(n)}{n^c} \int_{T}^{2T} \chi(1-c+it) (Xn)^{it} \ dt.
\end{align*}

The integral is in a form which may be handled by Lemma \ref{lem_JsigmaRAB} with $\sigma=1-c$ and $r~=~Xn$, where we note that a non-error-term contribution only arises when $2\pi Xn\in (T,2T]$, that is, $\frac{T}{2\pi X} < n \leq \frac{T}{\pi X}$. We therefore obtain
\begin{align*}
    I_{3,1}&=-\frac{X^{1-c}}{2\pi} \sum_{n=1}^\infty \frac{\Lambda(n)}{n^c}J(1-c,Xn,T)\\
    &=-X\sum_{\frac{T}{2\pi X} < n \le\frac{T}{\pi X}}\Lambda(n)e^{2\pi i Xn} + X^{1-c} \sum_{n=1}^\infty \frac{\Lambda(n)}{n^c} E(1-c,Xn,T).
\end{align*}

We now treat the error arising from $I_{3,1}$ in\ similar way to the treatment in Lemma~\ref{lem:RHS}, namely 
\begin{multline*}
    X^{1-c}\sum_{n=1}^\infty \frac{\Lambda(n)}{n^c} E(1-c,Xn,T) = O\left( X^{-1/\log T}\sum_{n=1}^\infty \frac{\Lambda(n)}{n^c} T^{1/2} \right) \\
    +O\left(X^{-1/\log T} \sum_{n=1}^\infty \frac{\Lambda(n)}{n^c}  \frac{T^{3/2}}{|T-2\pi X n|+T^{1/2}} \right).
\end{multline*}

As before, since $c=1+1/\log T$, the Dirichlet series can be evaluated and the first term is
\[
O\left( X^{-1/\log T} T^{1/2} \log T \right)
\]
and the second term can be bounded by
\begin{multline*}
X^{-1/\log T} \int_1^\infty n^{-1-1/\log T} \frac{T^{3/2}}{|T-2\pi X n|+T^{1/2}} \,dn \\
= T^{1/2} \int_{2\pi X/T}^\infty y^{-1-1/\log T} \frac{1}{|1-y|+T^{-1/2}} \,dy
\end{multline*}
where this time we substituted $y=2\pi X n / T$. If $2\pi X / T > 1+\epsilon$ then the integral is $\ll \int_1^\infty y^{-2-1/\log T} dy \ll 1$, and if $X/T = o(1)$ then there will be a contribution from $y\approx 0$ and from $y\approx 1$. The former contributes 
\[
\int_{2\pi X / T}^{1/2} y^{-1-1\-/\log T} dy \ll \log T \left(\frac{2\pi X}{T}\right)^{-1/\log T}  \ll \log T
\]
and the latter contributes 
\[
\int_{1/2}^{3/2} \frac{1}{|1-y|+T^{-1/2}} \,dy \ll \log T.
\]

These can all be combined uniformly into one error, $O(T^{1/2} \log T)$, which is an overestimate by $\log T$ in the case when $X > T^{\log \log T}$, but is a good bound for the size of $X$ required in our applications.
\end{proof}

\begin{lemma}\label{lem_I32}
    We have 
    \begin{align*}
        I_{3,2} &= -\frac{1}{2\pi}\int_{T}^{2T} \chi'(1-c+it) X^{1-c+it}\ dt \\
        &= \begin{cases}
        	X (\log X) e^{2\pi i X} + E'(X,T) & \text{ if } T < 2\pi X \leq 2T \\[1ex]
        	 E'(X,T) & \text{ if } X \leq \frac{T}{2\pi}  \text{ or if } X > \frac{T}{\pi} 
        \end{cases}
    \end{align*}
    with 
    \[
    E'(X,T) = O(T^{1/2} \log T) + O\left(\frac{T^{3/2} \log T}{|T-2\pi X|+T^{1/2}}\right) + O\left(\frac{T^{3/2}\log T}{|2T-2\pi X|+T^{1/2}}\right)
    \]
\end{lemma}

\begin{proof}[Proof of Lemma \ref{lem_I32}]
For $s=\sigma+it$ with $|\sigma|\le 2$ and $t>1$, we have 
\begin{equation*}
    \frac{\chi'}{\chi}(\sigma+it)=-\log\left(\frac{t}{2\pi}\right)+O\left(\frac{1}{t}\right),
\end{equation*}
so upon multiplying through by $\chi(\sigma+it)$ and using \eqref{stirling_absolutevalue} this yields \begin{equation*}
    \chi'(\sigma+it)=-\chi(\sigma+it)\log\left(\frac{t}{2\pi}\right)+O\left(t^{-1/2-\sigma}\right).
\end{equation*}

In our case we have $s=1-c+it$ and so we can write 
\begin{align*}
    I_{3,2}&=-\frac{1}{2\pi} \int_{T}^{2T} \left(-\chi(1-c+it)\log\left(\frac{t}{2\pi}\right) + O\left(t^{c-3/2}\right)\right)X^{1-c+it}\ dt\\
    &=\frac{1}{2\pi}\int_{T}^{2T} \chi(1-c+it)\log\left(\frac{t}{2\pi}\right) X^{1-c+it}\ dt + O\left(\int_{T}^{2T} X^{1-c}t^{c-3/2} dt\right)
\end{align*}

Upon integration we see that our error term may be bounded as $O(T^{1/2})$
which is no bigger than the error terms already found in $I_{3,1}$. Finally, a variant of Lemma~\ref{lem_JsigmaRAB} with $\sigma=1-c$ and $r=X$ yields

\begin{align*}
    I_{3,2}&=\frac{1}{2\pi} \int_{T}^{2T} \chi(1-c+it) \left(\log\frac{t}{2\pi}\right)X^{1-c+it}\, dt\\
    &= \begin{cases}
    	X (\log X) e^{2\pi i X} + E(1-c,X,T) \log T & \text{ if }  T < 2\pi X \leq 2T \\[1ex]
    	E(1-c,X,T) \log T & \text{ if } X \leq \frac{T}{2\pi}  \text{ or if } X > \frac{T}{\pi} 
    \end{cases}
\end{align*}
where $E(\sigma,r,T)$ is given in Lemma \ref{lem_JsigmaRAB}, and we have simply applied that lemma together with an integration by parts (which is the same argument used by Gonek~\cite[Lemma 3]{Gonek1984}).
\end{proof}

\section{Deduction of Shanks' conjecture}\label{sect:shanks1}
In this section we show how Theorem~\ref{mainthm} can be used to establish a new proof of Theorem~\ref{thm:shanks}, albeit with worse error terms. We begin by recalling the approximate functional equation for $\zeta'(\sigma+it)$ (see \cite[Lemma 1]{con}, for example), which states that if $0<\alpha<1$ and $s=\sigma+it$
\begin{equation}\label{eq:AFE}
    \zeta'(s)=-\sum_{n\le\left(\frac{t}{2\pi}\right)^{\alpha}}\frac{\log n}{n^s}+\chi(s)\sum_{n\le\left(\frac{t}{2\pi}\right)^{1-\alpha}}\frac{\log n - \ell(t)}{n^{1-s}}+O\left(t^{-\alpha/2}\log t\right) + O\left(t^{-(1-\alpha)/2}\log t\right)
\end{equation}
where $\ell(t) = \log\frac{t}{2\pi}$.

Taking the sum over non-trivial zeros $\rho=\frac{1}{2}+i\gamma$ with $0<\gamma\le T$, we have 
\begin{multline} \label{eq:zetaprimeA1A2A3}
    \sum_{0<\gamma\le T}\zeta'(\rho) = -\sum_{0<\gamma\le T}\sum_{n\le\left(\frac{\gamma}{2\pi}\right)^{\alpha}}\frac{\log n}{n^{\rho}}  +\sum_{0<\gamma\le T}\sum_{n\le\left(\frac{\gamma}{2\pi}\right)^{1-\alpha}} \frac{\chi(\rho) \log n}{n^{1-\rho}} \\
    -\sum_{0<\gamma\le T}\sum_{n\le\left(\frac{\gamma}{2\pi}\right)^{1-\alpha}} \frac{\chi(\rho) \ell(\gamma)}{n^{1-\rho}} + O\left( T^{1-\alpha/2} \log T\right) + O\left( T^{1/2+\alpha/2} \log T\right)
\end{multline}
and we label the three sums on the right-hand side $A_1$, $A_2$, and $A_3$ respectively.

We will see that if $0<\alpha<1$ these can all be summed using our results, but first we collect some standard results which we shall need for the evaluation of these sums.
\begin{lemma}\label{lem:partialsummations}
    We have the following asymptotic expansions:
    \begin{enumerate}
    \item \label{PSi} \begin{equation*}
    \sum_{n\le x}\frac{1}{n}=\log x+\gamma_0+O\left(\frac{1}{x}\right),
\end{equation*}
where $\gamma_0$ is Euler's constant.
\item  \label{PSii} \begin{equation*}
    \sum_{n\le x}\frac{\log n}{n}=\frac{1}{2}(\log x)^2+\gamma_1+O\left(\frac{\log x}{x}\right),
\end{equation*}
where the Stieltjes constant $\gamma_1$ is given in \eqref{eq:Laurent}.

\item  \label{PSiii} \begin{equation*}
        \sum_{n\le x}\frac{\Lambda(n)\log n}{n}=\frac{1}{2}(\log x)^2-(\gamma_0^2+2\gamma_1)+O\left(e^{-a\sqrt{\log x}}\right)
        \end{equation*}
        for some $a>0$, where the Stieltjes constants $\gamma_0, \gamma_1$ are given in \eqref{eq:Laurent}.

\item  \label{PSiv} For $C>-1$,
        \begin{equation*}
            \sum_{n\le x}n^{C}\log n=\frac{x^{C+1}\log x}{C+1}-\frac{x^{C+1}}{(C+1)^2}+O(x^{C}\log x).
        \end{equation*}
\item  \label{PSv} For $C>-1$,
        \begin{equation*}
            \sum_{n\le x}n^{C}\Lambda(n)\log n=\frac{x^{C+1}\log x}{C+1}-\frac{x^{C+1}}{(C+1)^2}+O(x^{C+1}e^{-a\sqrt{\log x}}),
        \end{equation*}
         for some $a>0$.
\end{enumerate}
\end{lemma}

\begin{proof}We note that parts (\ref{PSi}) and (\ref{PSii}) are standard and can be found in \cite[p.55,70]{apostol} respectively. Part (\ref{PSiii}) is an application of Perron's formula on $\left(\frac{\zeta'}{\zeta}(s)\right)'$. We may obtain (\ref{PSiv}) using partial summation with $a_n=\log n, f(t)=t^C$. From this, another partial summation with $a_n=\Lambda(n)$ and $f(t)=n^C\log n$ establishes (\ref{PSv}).
\end{proof}

We handle $A_1$ first. By swapping the order of summation we may write 
\begin{align*}
    A_1&=-\sum_{0<\gamma\le T}\sum_{n\le\left(\frac{\gamma}{2\pi}\right)^{\alpha}}\frac{\log n}{n^{\rho}}\\
    &=-\sum_{n\le\left(\frac{T}{2\pi}\right)^{\alpha}} \log n \sum_{2\pi n^{1/\alpha}\le\gamma\le T} \frac{1}{n^{\rho}}.
\end{align*}

Using the Landau-Gonek formula, \eqref{eq:SumNToTheRho}, we have 
\begin{align}
   A_1 &= -\sum_{n\le\left(\frac{T}{2\pi}\right)^{\alpha}}\log n\left(- \frac{T}{2\pi n}\Lambda(n) + \Lambda(n)  n^{-1+1/\alpha} + O(\log T \log\log T)\right) \notag\\
   &= \frac{T}{2\pi} \sum_{n\le\left(\frac{T}{2\pi}\right)^{\alpha}} \frac{\Lambda(n) \log n}{n}  - \sum_{n\le\left(\frac{T}{2\pi}\right)^{\alpha}}  n^{-1+1/\alpha}\Lambda(n)\log n  + O\left(T^{\alpha} (\log T)^2 \log\log T\right) \notag\\
   &=\frac{T}{4\pi}\alpha^2\left(\log\frac{T}{2\pi}\right)^2-\frac{T}{2\pi}\alpha^2\left(\log\frac{T}{2\pi}\right)+\frac{T}{2\pi}(-\gamma_0^2-2\gamma_1+\alpha^2) + O\left(T e^{-a\sqrt{\log T}}\right), \label{eq:A1}
\end{align}
having applied Lemma \ref{lem:partialsummations} parts (\ref{PSiii}) and (\ref{PSv}) with the choices $x=\left(\frac{T}{2\pi}\right)^{\alpha}$ and $C=-1+1/\alpha$ (and using that trivially $x^{1/\alpha}=\frac{T}{2\pi}$ and $\log x=\alpha\left(\log\frac{T}{2\pi}\right)$). Note that we need $\alpha<1$ for the claimed error term in the last line (coming from summing the von Mangoldt function) to dominate the error term in the middle line (coming from the Landau--Gonek formula).

Now we turn to $A_2$. We have 
\begin{align}
    A_2 &=\sum_{0<\gamma\le T}\sum_{n\le\left(\frac{\gamma}{2\pi}\right)^{1-\alpha}}\frac{\chi(\rho)\log n}{n^{1-\rho}} \notag\\
    &=\sum_{n\le\left(\frac{T}{2\pi}\right)^{1-\alpha}}\frac{\log n}{n} \sum_{2\pi n^{1/(1-\alpha)}\le\gamma\le T}\chi(\rho)n^{\rho} \notag\\
    &=-\frac{T}{2\pi}\sum_{n\le\left(\frac{T}{2\pi}\right)^{1-\alpha}}\frac{\log n}{n} + \sum_{n\le\left(\frac{T}{2\pi}\right)^{1-\alpha}}n^{-1+1/(1-\alpha)}\log n \notag\\
    &\qquad+ \sum_{n \leq \left(\frac{T}{2\pi}\right)^{1-\alpha}} \left( O\left( (\log n)^2 \right) + O\left( \frac{\log n}{n} T e^{-a\sqrt{\log T}}  \right) \right) \notag\\
    &=-\frac{T}{4\pi}(1-\alpha)^2\left(\log\frac{T}{2\pi}\right)^2+\frac{T}{2\pi}(1-\alpha)^2\left(\log\frac{T}{2\pi}\right)
    -\frac{T}{2\pi}(\gamma_1+(1-\alpha)^2) \notag\\
    &\qquad + O\left( T e^{-a'\sqrt{\log T}}  \right), \label{eq:A2}
\end{align}
where in passing from the first line to the second we have swapped the order of summation, in passing from the second to the third line we have applied Corollary \ref{mainthmwithinteger} (with the unconditional error term). To get from the third line to the fourth, we have applied Lemma \ref{lem:partialsummations} parts (\ref{PSii}) and (\ref{PSiv}) with $x=\left(\frac{T}{2\pi}\right)^{1-\alpha}$ and $C=-1+\frac{1}{1-\alpha}$. Note that we must have $\alpha>0$ for the claimed error term to be dominant when passing from the third to the fourth line.

Finally we turn to $A_3$. We have 
\begin{align*}
    A_3 &=-\sum_{0<\gamma\le T}\sum_{n\le\left(\frac{\gamma}{2\pi}\right)^{1-\alpha}}\chi(\rho)\frac{\log\left(\frac{\gamma}{2\pi}\right)}{n^{1-\rho}}\\
    &=-\sum_{n\le\left(\frac{T}{2\pi}\right)^{1-\alpha}}\frac{1}{n}\sum_{2\pi n^{1/(1-\alpha)}\le\gamma\le T}\chi(\rho)n^{\rho}\log\frac{\gamma}{2\pi} ,
\end{align*}
by swapping the order of summation. 
By Abel summation and Corollary~\ref{mainthmwithinteger} we can easily evaluate this inner sum using
\begin{align*}
\sum_{\gamma\le T}\chi(\rho)n^{\rho}\log \frac{\gamma}{2\pi} &= \int_2^T \log\left(\frac{t}{2\pi}\right) \ dS(n,t) \\
&= \log\frac{T}{2\pi} S(n,T) - \int_2^T \frac{1}{t} S(n,t) \ dt \\
&=- \frac{T}{2\pi} \log \frac{T}{2\pi}+\frac{T}{2\pi} + O\left(T e^{-a \sqrt{\log T}} \right),
\end{align*}
which we obtain simply by substituting in Corollary~\ref{mainthmwithinteger} and simplifying. Note that since $n=O(T^{1-\alpha})$ with $\alpha>0$, the claimed error term is the dominant one of the two in the Corollary.
Therefore, by substituting this expression into the definition of $A_3$ and subtracting the terms with $\gamma\le 2\pi n^{\frac{1}{1-\alpha}}$ we deduce
\begin{align*}
A_3 &= -\sum_{n\le\left(\frac{T}{2\pi}\right)^{1-\alpha}}\frac{1}{n}\left(-\frac{T}{2\pi}\log\frac{T}{2\pi}+\frac{T}{2\pi}+\frac{1}{1-\alpha}n^{\frac{1}{1-\alpha}}\log n-n^{\frac{1}{1-\alpha}} + O\left(T e^{-a \sqrt{\log T}} \right) \right),
\end{align*}
and once again applying Lemma \ref{lem:partialsummations} gives \begin{multline}
    A_3=\frac{T}{2\pi}(1-\alpha)\left(\log\frac{T}{2\pi}\right)^2+\frac{T}{2\pi}(\gamma_0-2(1-\alpha))\left(\log\frac{T}{2\pi}\right)-\frac{T}{2\pi}(\gamma_0+2(\alpha-1))\\
    + O\left(T e^{-a' \sqrt{\log T}} \right) \label{eq:A3}
\end{multline} 
for any $0<a'<a$.

We combine $A_1$, $A_2$ and $A_3$ from \eqref{eq:A1}, \eqref{eq:A2}, and \eqref{eq:A3} respectively into \eqref{eq:zetaprimeA1A2A3}, to obtain 
\begin{multline*}
    \sum_{0<\gamma\le T}\zeta'(\rho) 
    =\frac{T}{4\pi}\left(\log\frac{T}{2\pi}\right)^2+(\gamma_0-1)\frac{T}{2\pi}\left(\log\frac{T}{2\pi}\right)+(1-\gamma_0-\gamma_0^2-3\gamma_1)\frac{T}{2\pi} \\
    + O\left( T e^{-a\sqrt{\log T}}  \right)
\end{multline*}
for some $a>0$. This recovers Theorem \ref{thm:shanks} with the unconditional error term. 

\begin{rem*}
    If we had assumed the Riemann Hypothesis, then the error terms in Lemma~\ref{lem:partialsummations} would be smaller (roughly $x^{1/2}$) and the resulting error terms would look like $T^{\alpha}$ and $T^{1-\alpha}$ times certain powers of logarithms, and thus be optimised when  the approximate functional equation has roughly equal weight in both its pieces, i.e.~when $\alpha=1/2$.
\end{rem*}

\section{Deduction of the generalised Shanks' conjecture}\label{sect:shanks2}
The purpose of this section is to deduce Theorem \ref{higherderivs} from Theorem \ref{mainthm}. This follows essentially the same approach as that in Section \ref{sect:shanks1}. We firstly start by recalling the approximate functional equation for $\zeta^{(\nu)}(\sigma+it)$ (see, for example, Equation (36) of \cite{con}), which says that 
\begin{multline*}
    \zeta^{(\nu)}(s)=(-1)^{\nu}\sum_{n\le\left(\frac{t}{2\pi}\right)^{\alpha}}\frac{(\log n)^{\nu}}{n^s} +\chi(s)\sum_{n\le\left(\frac{t}{2\pi}\right)^{1-\alpha}}\frac{(\log n-\ell(t))^{\nu}}{n^{1-s}} \\+ O\left(t^{-\alpha/2}(\log t)^{\nu+1}\right)+O\left(t^{-(1-\alpha)/2}(\log t)^{\nu+1}\right)
\end{multline*}
for $0<\alpha<1$, where $\ell(t) = \log\frac{t}{2\pi}$.

Before proceeding, we shall need the following lemma which is analogous to Lemma~\ref{lem:partialsummations}.

\begin{lemma}\label{lem:generalised_partial_summations}
    We have the following asymptotic expansions:  For $\nu\ge 0$ and $C>-1$,
    \begin{enumerate}
        \item \label{lemparti}
        \begin{equation*}
            \sum_{n\le x}\frac{\Lambda(n)(\log n)^{\nu}}{n}=\frac{(\log x)^{\nu+1}}{\nu+1}+O((\log x)^{\nu}).
        \end{equation*}
        
        \item \label{lempartii}
        \begin{equation*}
            \sum_{n\le x} \Lambda(n) n^{C} (\log n)^{\nu} = \frac{x^{C+1} (\log x)^\nu}{C+1} + O\left(x^{C+1}(\log x)^{\nu-1}\right).
        \end{equation*}

        \item \label{lempartiii}
        \begin{equation*}
        \sum_{n\le x}\frac{(\log n)^\nu}{n}=\frac{(\log x)^{\nu+1}}{\nu+1}+O((\log x)^\nu)
        \end{equation*}

        \item \label{lempartiv}
        \begin{equation*}
            \sum_{n\le x}n^{C}(\log n)^{\nu} = \frac{x^{C+1} (\log x)^\nu}{C+1} + O\left(x^{C+1}(\log x)^{\nu-1}\right).
        \end{equation*}

    \end{enumerate}
\end{lemma}

\begin{proof}[Proof of Lemma \ref{lem:generalised_partial_summations}]
Part (\ref{lemparti}) may be found in \cite[p.30]{Karabulut}, from which Part (\ref{lempartii}) follows by partial summation. Parts  (\ref{lempartiii}) and (\ref{lempartiv}) are standard applications of the Euler--Maclaurin summation formula.
\end{proof}

Taking the sum over non-trivial zeros $\rho=\beta+i\gamma$ with $0<\gamma\le T$, we have 
\begin{multline*}
    \sum_{0<\gamma\le T}\zeta^{(\nu)}(\rho) = (-1)^{\nu}\sum_{0<\gamma\le T}\sum_{n\le\left(\frac{\gamma}{2\pi}\right)^{\alpha}}\frac{(\log n)^{\nu}}{n^{\rho}}  \\
    +\sum_{0<\gamma\le T}\sum_{n\le\left(\frac{\gamma}{2\pi}\right)^{1-\alpha}}\frac{\chi(\rho)}{n^{1-\rho}}(\log n-\ell(\gamma))^{\nu}  \\+ O\left(T^{1-\alpha/2}(\log T)^{\nu+1}\right)+O\left(T^{1/2+\alpha/2}(\log T)^{\nu+1}\right)
\end{multline*}
with $0<\alpha<1$. We label the two sums on the right-hand side $B_1$, and $B_2$, respectively.

For the first term $B_1$, we may use the Landau--Gonek formula, \eqref{eq:SumNToTheRho}, and we have
\begin{align*}
    B_1 &= (-1)^{\nu} \sum_{0<\gamma\le T} \sum_{n\le\left(\frac{\gamma}{2\pi}\right)^{\alpha}}\frac{(\log n)^{\nu}}{n^{\rho}}\\
    &= (-1)^{\nu} \sum_{n\le\left(\frac{T}{2\pi}\right)^{\alpha}} (\log n)^{\nu} \sum_{2\pi n^{1/\alpha} \le \gamma\le T} \frac{1}{n^{\rho}}\\
    &= (-1)^{\nu} \sum_{n\le\left(\frac{T}{2\pi}\right)^{\alpha}}(\log n)^{\nu}\left(-\frac{T}{2\pi n}\Lambda(n) + \Lambda(n) n^{-1+1/\alpha} + O(\log T\log\log T)\right)\\
    &= (-1)^{\nu+1}\frac{T}{2\pi}\sum_{n\le\left(\frac{T}{2\pi}\right)^{\alpha}}\frac{\Lambda(n)(\log n)^{\nu}}{n} + (-1)^{\nu} \sum_{n\le\left(\frac{T}{2\pi}\right)^{\alpha}}\Lambda(n)(\log n)^{\nu}n^{-1+1/\alpha}\\
    &\qquad + O\left( T^\alpha (\log T)^{\nu+1} \log\log T\right)
\end{align*}

It follows immediately from part (\ref{lemparti}) of Lemma \ref{lem:generalised_partial_summations} that
\begin{equation*}
     \frac{T}{2\pi}\sum_{n\le\left(\frac{T}{2\pi}\right)^{\alpha}}\frac{\Lambda(n)(\log n)^{\nu}}{n} = \frac{\alpha^{\nu+1}}{\nu+1} \frac{T}{2\pi} \left(\log\frac{T}{2\pi}\right)^{\nu+1}+O\left(T(\log T)^\nu \right).
\end{equation*}
Moreover, from part (\ref{lempartii}) of the same lemma we see that
\begin{equation*}
    \sum_{n\le\left(\frac{T}{2\pi}\right)^{\alpha}}\Lambda(n)(\log n)^{\nu}n^{-1+1/\alpha} = O(T(\log T)^{\nu})
\end{equation*}

Therefore, if $\alpha<1$,
\begin{equation}\label{eq:B1}
B_1=(-1)^{\nu+1} \frac{\alpha^{\nu+1}}{\nu+1} \frac{T}{2\pi} \left(\log\frac{T}{2\pi}\right)^{\nu+1}+O(T(\log T)^{\nu}).
\end{equation}

For $B_2$ we  use the binomial theorem and swap the order of summation to write
\begin{align*}
    B_2&=\sum_{0<\gamma\le T}\sum_{n\le\left(\frac{\gamma}{2\pi}\right)^{1-\alpha}}\frac{\chi(\rho)}{n^{1-\rho}}(\log n-\ell(\gamma))^{\nu}\\
    &=\sum_{j=0}^{\nu}\binom{\nu}{j}(-1)^{\nu-j}\sum_{0<\gamma\le T}\sum_{n\le\left(\frac{\gamma}{2\pi}\right)^{1-\alpha}}\frac{\chi(\rho)(\log n)^j\ell(\gamma)^{\nu-j}}{n^{1-\rho}} \\
    &= \sum_{j=0}^{\nu}\binom{\nu}{j} (-1)^{\nu-j} \sum_{n\le\left(\frac{T}{2\pi}\right)^{1-\alpha}} \frac{(\log n)^j}{n} \sum_{2\pi n^{1/(1-\alpha)}\le\gamma\le T} \chi(\rho) n^{\rho} \ell(\gamma)^{\nu-j}
\end{align*}

We may handle the innermost sum in a similar way as we did for $\nu=1$, to obtain 
\begin{multline*}
    \sum_{2\pi n^{1/(1-\alpha)}\le\gamma\le T}\chi(\rho)n^{\rho}\left(\log\frac{\gamma}{2\pi}\right)^{\nu-j} = \left(\log\frac{T}{2\pi}\right)^{\nu-j} S(n,T) + \frac{ (\log n)^{\nu-j}}{(1-\alpha)^{\nu-j}} S\left(n,2\pi n^{1/(1-\alpha)}\right) \\
    - (\nu-j) \int_{2\pi n^{1/(1-\alpha)}}^T \frac{\left(\log \frac{t}{2\pi} \right)^{\nu-j-1}}{t}  S(n,t) \ dt 
    \end{multline*}
and using Corollary~\ref{mainthmwithinteger} to evaluate $S(n,T)$ we see this is, for $\alpha>0$,
\begin{multline*}
\sum_{2\pi n^{1/(1-\alpha)}\le\gamma\le T}\chi(\rho)n^{\rho}\left(\log\frac{\gamma}{2\pi}\right)^{\nu-j} = -\frac{T}{2\pi}\left(\log\frac{T}{2\pi}\right)^{\nu-j} +  \frac{ n^{1/(1-\alpha)}(\log n)^{\nu-j}}{(1-\alpha)^{\nu-j}} \\  + O\left(T (\log T)^{\nu-j-1}\right)
\end{multline*}

Part (\ref{lempartiii}) of Lemma~\ref{lem:generalised_partial_summations} shows
\[
\sum_{n\le\left(\frac{T}{2\pi}\right)^{1-\alpha}} \frac{(\log n)^j}{n} \frac{T}{2\pi}\left(\log\frac{T}{2\pi}\right)^{\nu-j} = \frac{(1-\alpha)^{j+1}}{j+1} \frac{T}{2\pi} \left(\log\frac{T}{2\pi}\right)^{\nu+1}
\]
and Part (\ref{lempartiv}) shows
\[
\sum_{n\le\left(\frac{T}{2\pi}\right)^{1-\alpha}}  \frac{ n^{-1+1/(1-\alpha)}(\log n)^{\nu}}{(1-\alpha)^{\nu-j}} = O\left(T (\log T)^\nu \right)
\]

We deduce that for $\alpha>0$
\begin{equation*}
    B_2=(-1)^{\nu+1}\sum_{j=0}^{\nu}(-1)^{j}\binom{\nu}{j}\frac{(1-\alpha)^{j+1}}{j+1} \frac{T}{2\pi} \left(\log\frac{T}{2\pi}\right)^{\nu+1}+O(T(\log T)^{\nu}).
\end{equation*}

It follows from the binomial identity that
\[
\sum_{j=0}^{\nu}(-1)^{j}\binom{\nu}{j}\frac{(1-\alpha)^{j+1}}{j+1} = \frac{1-\alpha^{\nu+1}}{\nu+1}
\]

This allows us to simplify $B_2$ as \begin{equation*}
    B_2=(-1)^{\nu+1} \frac{1-\alpha^{\nu+1}}{\nu+1} \frac{T}{2\pi} \left(\log\frac{T}{2\pi}\right)^{\nu+1}+O(T(\log T)^{\nu}).
\end{equation*}

Finally since $\sum_{0<\gamma\le T}\zeta^{(\nu)}(\rho)=B_1+B_2$ and inserting $B_1$ from \eqref{eq:B1}, we see that 
\begin{equation*}
    \sum_{0<\gamma\le T}\zeta^{(\nu)}(\rho) = \frac{(-1)^{\nu+1}}{\nu+1} \frac{T}{2\pi} \left(\log\frac{T}{2\pi}\right)^{\nu+1}+O(T(\log T)^{\nu}),
\end{equation*}
which recovers Theorem \ref{higherderivs}.

\addcontentsline{toc}{chapter}{Bibliography}
\bibliography{references}        
\bibliographystyle{plain}  
\end{document}